\documentclass[12pt]{amsart}

\usepackage{amsmath,amssymb,amsfonts}
\usepackage{hyperref}
\usepackage{graphicx}
\usepackage[dvips]{pstricks}
\usepackage{pst-3dplot}
\psset{arrowsize=4pt}

\newtheorem{theorem}{Theorem}[section]
\newtheorem{lemma}[theorem]{Lemma}
\newtheorem{proposition}[theorem]{Proposition}

\theoremstyle{definition}
\newtheorem{definition}[theorem]{Definition}
\newtheorem{example}[theorem]{Example}

\theoremstyle{remark}
\newtheorem{remark}[theorem]{Remark}
\newtheorem{notation}[theorem]{Notation}

\newcommand {\R} {\ensuremath {\mathbb{R}} }
\newcommand {\N} {\ensuremath {\mathbb{N}} }
\newcommand {\isom} {\ensuremath {\cong} }

\newcommand {\onto} {\twoheadrightarrow}
\newcommand {\isomto} {\ensuremath {\xrightarrow{\isom}} }
\newcommand {\hto} {\ensuremath {\xrightarrow{\simeq}} }
\newcommand {\hfrom} {\ensuremath {\xleftarrow{\simeq}} }
\newcommand {\xto}[1] {\ensuremath {\xrightarrow{#1}} }
\newcommand {\Top} {\ensuremath {\mathbf{Top}} }

\newcommand {\adjn} {\leftrightarrows}

\newcommand {\dI} {\vec{I}}

\newcommand{\fc}{\vec{\pi}_1} 
\newcommand{\Cat} {\ensuremath{\mathbf{Cat}}}
\newcommand{\C}{\mathbf{C}}
\newcommand{\D} {\mathbf{D}}

\newcommand{\dTop}{\mathbf{dTop}}
\newcommand{\dS}{\vec{S}^1}

\DeclareMathOperator{\Id}{Id}
\DeclareMathOperator{\im}{im}

\DeclareMathOperator{\Int}{Int}
\DeclareMathOperator{\Ext}{Extrl}

\newgray{gray1}{0.6}
\newgray{gray2}{0.5}

\begin{document}

\title[Models and van Kampen theorems for DHT]{Models and van Kampen theorems for directed homotopy theory}

\author{Peter Bubenik}
\email{p.bubenik@csuohio.edu}
\address{Department of Mathematics\\
  Cleveland State University\\
  2121 Euclid Ave. RT 1515\\
  Cleveland OH, 44115-2214\\
  USA}

\date{\today}


\begin{abstract}
  We study topological spaces with a
  distinguished set of paths, called directed paths. Since these
  directed paths are generally not reversible, the directed homotopy
  classes of directed paths do not assemble into a groupoid, and there
  is no direct analog of the fundamental group. However, they do
  assemble into a category, called the fundamental category. We
  define models of the fundamental category, such as the fundamental
  bipartite graph, and minimal extremal models which are shown to generalize
  the fundamental group. In addition, we prove van Kampen theorems for
  subcategories, retracts, and models of the fundamental category.
\end{abstract}
 
\maketitle

\section{Introduction} 
\label{sec:introduction}

\subsection{Directed spaces and directed homotopies}

The field of directed algebraic topology studies \emph{directed spaces}. That is, topological spaces together with a (local) order, or more generally, spaces together with a subset of allowed paths, called \emph{directed paths}. In either approach, the directed paths are generally not reversible. Consequently, the directed homotopy classes of directed paths behave much differently from the usual homotopy classes of paths (see Example~\ref{example:dipaths}). 
As many topologists are unfamiliar with directed algebraic topology, we give a leisurely introduction, which includes the main new constructions and results of this paper. 

A motivation for this study comes from the field of concurrent (parallel) computing, in which multiple processes have access to shared resources. A directed space models the state space of such a system, and the directed paths model the execution paths. General relativity provides another possible application. For more details, the reader is referred to the papers~\cite{goubault:someGPiCT,frg:atc}.

A number of categorical settings have been used to develop
directed algebraic topology. These include partially ordered spaces
(pospaces) \cite{fghr:components,bubenik:context}, local pospaces
\cite{frg:atc, bubenikWorytkiewicz:modelCategoryForLPS,
  worytkiewicz:lps}, preordered spaces~\cite{grandis:shape}, local
preordered spaces \cite{krishnan:streams}, \mbox{d--spaces}
\cite{grandis:dht1,raussen:invariants}, flows
\cite{gaucher:modelCategory}, and cubical complexes (also known as
higher--dimensional automata)
\cite{pratt:modelingConcurrencyWithGeometry,gaucherGoubault:hda,fajstrup:cubicalComplex,fajstrup:cubicalLPO}.
Here we work in the general setting of Grandis' \emph{\mbox{d--spaces}}.
\begin{definition}[\cite{grandis:dht1}]
  A \emph{\mbox{d--space}} is a topological space $X$ together with a set $dX$
  of paths $\gamma: [0,1] \to X$, called \emph{directed paths} or
  \emph{dipaths} satisfying the following axioms:
  \begin{enumerate}
  \item for all $x \in X$, the constant path $c_x(t) = x$ is in $dX$,
  \item \label{en:reparam} $dX$ is closed under reparametrization: if
    $\gamma \in dX$ and $f:[0,1] \to [0,1]$ is continuous and
    non-decreasing then $\gamma \circ f \in dX$, and
  \item $dX$ is closed under concatenation: if $\gamma_1, \gamma_2 \in
    dX$ and $\gamma_1(1) = \gamma_2(0)$, then $\gamma \in dX$ where
    $\gamma(t) = \gamma_1(2t)$, for $0\leq t\leq \frac{1}{2}$, and
    $\gamma(t) = \gamma_2(2t-1)$, for $\frac{1}{2} \leq t \leq 1$.
  \end{enumerate}
\end{definition}
Since $f:[0,1] \to [0,1]$ in \eqref{en:reparam} above need not be onto,
subpaths of dipaths are also dipaths.  A morphism of \mbox{d--spaces}
$f:X \to Y$, called a \emph{dimap}, is a continuous map which
preserves dipaths. That is, $f(dX) \subseteq dY$, where $f(\gamma) = f
\circ \gamma$.

\begin{example}
  \begin{itemize}
  \item Any topological space $X$ is a \mbox{d--space} with $dX$ equal to the set of all paths in $X$.
  \item Let $\dI = (I, dI)$ where $I$ is the closed interval $[0,1]$ and $dI$ is the set of all non--decreasing continuous maps $I \to I$.
Dipaths in \mbox{d--space} $X$ coincide with dimaps $\dI \to X$.
  \item Let $\dS$ be the unit circle together with all counterclockwise paths.
  \item Given two \mbox{d--spaces} $X$ and $Y$, then $X \times Y$ is a \mbox{d--space} with $d(X \times Y)$ = $dX \times dY$ where $(f,g)(t) = (f(t),g(t))$.
  \item If $X$ is a \mbox{d--space} and $A \subseteq X$, then $A$ is a \mbox{d--space} with $dA$ equal to the subset of paths in $dX$ whose image is in $A$.
  \end{itemize}
\end{example}

An advantage of using \mbox{d--spaces} over the more commonly used preordered spaces is that we can model loops, such as with $\dS$.

A \emph{\mbox{d--homo}topy} between dimaps $f,g:X \to Y$ is a dimap
$H:X \times \dI \to Y$ such that for all $x \in X$, $H(x,0) = f(x)$,
and $H(x,1) = g(x)$. We write $H:f \hto g$ and $H_0 = f$ and $H_1 =
g$. Notice that this notion is not symmetric. To obtain an equivalence
relation we take the transitive symmetric closure and say that $f$ is
\emph{\mbox{d--homo}topic} to $g$ if they are linked by a chain of
\mbox{d--homo}topies, $f \hto f_1 \hfrom f_2 \hto \ldots \hto g$.

Given a dipath $\gamma$ from $a$ to $b$, let $[\gamma]$ denote the equivalence class of dimaps
$\gamma': \dI \to X$ with $\gamma'(0)=a$, $\gamma'(1)=b$ and $\gamma'$
\mbox{d--homo}topic to $\gamma$ relative to $\{a,b\}$. That is, we insist that
the \mbox{d--homo}topies linking $\gamma'$ and $\gamma$ leave the endpoints
fixed. Call this a directed homotopy class of directed paths from $a$
to $b$, or more simply, a homotopy class of dipaths.

\begin{example} \label{example:dipaths}
  The directed paths up to directed homotopy of a \mbox{d--space} are very
  different from the paths up to homotopy of the underlying
  topological space.

  For example, an undirected path $\gamma: I \to X$ need not be
  homotopic relative to its endpoints to a directed path in a \mbox{d--space}
  $X$. Consider the following example, which is a subspace obtained
  from $\dI \times \dI$ by removing two squares.
\begin{center}
\psset{unit=2cm}
\begin{pspicture}(1,1)
\psset{fillstyle=solid,fillcolor=gray,linecolor=black}
\psframe(0,0)(1,1)
\psset{fillcolor=white,linecolor=black}
\psframe(0.2,0.6)(0.4,0.8)
\psframe(0.6,0.2)(0.8,0.4)
\psset{linecolor=black,fillstyle=none}
\psline{->}(0.3,0)(0.3001,0)
\psline{->}(0,0.55)(0,0.5501)
\psline(0,0)(1,1)
\psline{->}(0,0)(0.5,0.5)
\psbezier(0,0)(0.5,0)(0.7,0)(0.8,0.1)
\psbezier(0.8,0.1)(1,0.3)(1,0.9)(1,1)
\psline{->}(0.8,0.1)(0.8001,0.1001)
\psbezier(0,0)(0,0.5)(0,0.7)(0.1,0.8)
\psbezier(0.1,0.8)(0.3,1)(0.9,1)(1,1)
\psline{->}(0.1,0.8)(0.1001,0.8001)
\pscurve[linecolor=yellow](0,0)(0.05,0.05)(0.15,0.85)(0.45,0.85)(0.55,0.15)(0.85,0.15)(0.95,0.95)(1,1)
\end{pspicture}
\end{center}
Furthermore, directed paths in a space that is contractible in the undirected sense are not necessarily \mbox{d--homo}topic. In the following figure we have two non-homotopic dipaths in a contractible \mbox{d--space} obtained from $\dI \times \dI \times \dI$ by removing two isothetic parallelepipeds which intersect the boundary of $X$.
  \begin{center}
    \includegraphics[width=7cm]{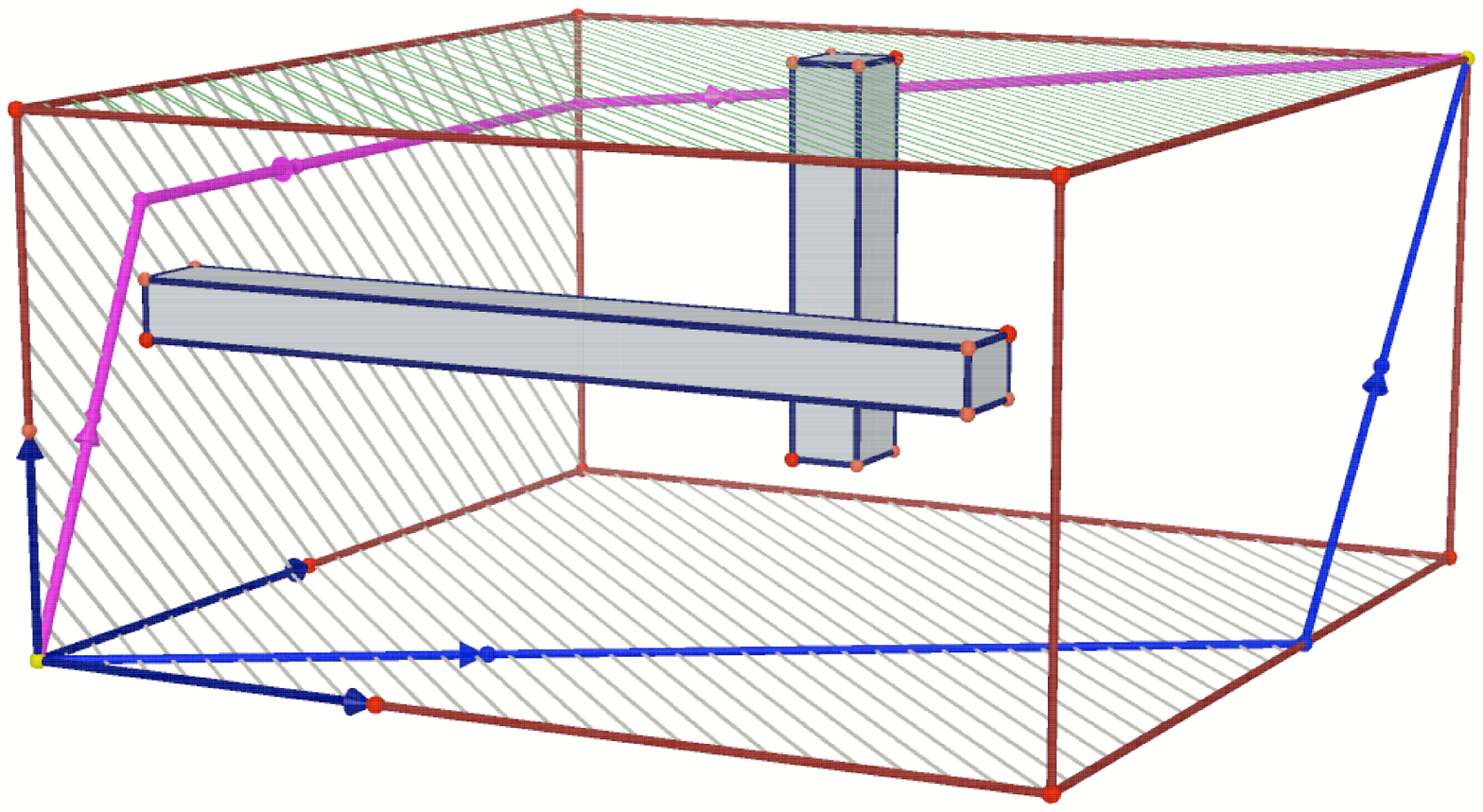}
  \end{center}
\end{example}

\subsection{The fundamental category}

In trying to understand the directed paths in a directed space, $X$, a
basic object of study is the \emph{fundamental category},
$\fc(X)$. Its objects are the points in $X$, and for $a,b \in X$, the
morphisms $\fc(X)(a,b)$ are given by the directed homotopy classes of
directed paths from $a$ to $b$. The undirected version of this
definition results in the fundamental groupoid, in which all morphisms
are invertible. When $X$ is a \mbox{d--space} the only invertible
morphisms in the fundamental category are the homotopy classes for
reversible dipaths.

In usual undirected algebraic topology, the fundamental groupoid is
often simplified to the fundamental group by identifying the
isomorphism classes of objects. That is, the fundamental group is the
\emph{skeleton} of the fundamental groupoid. However, for
\mbox{d--spaces} where the only reversible paths are the constant
paths, the fundamental category is its own skeleton.

This is a central difficulty, and has led to considerable research in
directed algebraic topology. The goal is to reduce the fundamental
category, which typically has uncountably many objects, to some
considerably smaller and preferably finite structure that still
contains `the essential information'.

One approach, explored by Fajstrup, Goubault, Haucourt, and
Raussen \cite{fghr:components,goubaultHaucourt:components2, haucourt:loopFree}, is to use
the calculus of fractions or generalized equivalences to reduce the
fundamental category to its component category. Here we follow
Grandis' approach~\cite{grandis:shape, grandis:quotient} and look for
a (possibly finite) full subcategory of the fundamental category, that will provide an adequate model of the fundamental category.

Now we introduce some new notation that will be useful. Similar notation has been used for fundamental groupoids.
\begin{notation}
  Let $A \subseteq X$ be a subspace. Let $\fc(X,A)$ denote the full subcategory of $\fc(X)$ generated by $A$. That is, $\fc(X,A)$ has as objects the points in $A$, and for $a,b \in A$, $\fc(X,A)(a,b) = \fc(X)(a,b)$.
Let $\iota: \fc(X,A) \to \fc(X)$ denote the inclusion.
For $x \in X$ we simplify $\fc(X,\{x\})$ to $\fc(X,x)$.
\end{notation}

\subsection{Fundamental bipartite graphs}

We introduce a new full subcategory of the fundamental category that is useful for many of the \mbox{d--spaces} that appear in applications.

\begin{definition}
  The objects of a category $\C$ have a preorder defined by $x \leq y$ iff there exists a morphism from $x$ to $y$.
  Call an object $a \in \C$ \emph{minimal} if $x \leq a$ implies
  $x=a$. Similarly define $b \in \C$ to be \emph{maximal} if $b\leq x$
  implies $b=x$. Say that an object is \emph{extremal} if it is either
  maximal or minimal. Let $\Ext(\C)$ denote the set of
  all extremal objects in $\C$.
  For a \mbox{d--space} $X$ we will sometimes let $\Ext(X)$ denote
  $\Ext(\fc(X))$.  Define the \emph{fundamental bipartite graph} of
  $X$ to be $\fc(X,\Ext(X))$.  To view this category as a bipartite
  graph, we ignore the identity maps.
\end{definition}

\begin{example}
Let $X$ be the subspace of $\dI \times \dI$ in the left--hand figure. Its fundamental bipartite graph has two vertices and four edges. We remark that the branching information is lost in this graph.
\begin{center}
\psset{unit=2cm}
\begin{pspicture}(-0.2,-0.1)(1.2,1.1)
\psset{fillstyle=solid,fillcolor=gray,linecolor=black}
\psframe(0,0)(1,1)
\psset{fillcolor=white,linecolor=black}
\psframe(0.2,0.2)(0.4,0.4)
\psframe(0.6,0.6)(0.8,0.8)
\psset{linecolor=black,fillstyle=none}
\psline{->}(0.3,0)(0.3001,0)
\psline{->}(0,0.55)(0,0.5501)
\psdots(0,0)(1,1)
\uput[l](0,0){$a$}
\uput[r](1,1){$b$}
\end{pspicture}
\quad \quad \quad \quad
\psset{unit=2cm}
\begin{pspicture}(-0.2,-0.1)(1.2,1.1)
\psdots(0,0.5)(1,0.5)
\uput[l](0,0.5){$a$}
\uput[r](1,0.5){$b$}
\psset{linestyle=dashed,linecolor=red}
\pscurve(0,0.5)(0.5,0.35)(1,0.5)
\pscurve(0,0.5)(0.5,0.45)(1,0.5)
\pscurve(0,0.5)(0.5,0.55)(1,0.5)
\pscurve(0,0.5)(0.5,0.65)(1,0.5)
\psline{->}(0.53,0.35)(0.531,0.35)
\psline{->}(0.53,0.45)(0.531,0.45)
\psline{->}(0.53,0.55)(0.531,0.55)
\psline{->}(0.53,0.65)(0.531,0.65)
\end{pspicture}
\end{center}
\end{example}

\subsection{Past retracts and future retracts}

In order to simplify the fundamental category, one obvious approach to
the homotopy theorist is to apply directed homotopies to the
underlying space.
\begin{definition}
 Call a directed map $H: X \times \dI \to X$ a
\emph{future homotopy flow} if $H_0 = \Id_X$ and a \emph{past homotopy
  flow} if $H_1 = \Id_X$. For a future (past) homotopy flow let $f$
equal $H_1$ ($H_0$). A future (past) homotopy flow induces a functor
$\fc(X) \to \fc(\im f) \isom \fc(X,\im f)$.
\end{definition}
Raussen~\cite{raussen:invariants} has carefully studied these flows.

A fruitful generalization at the level of the fundamental category is
given by the following definition. In Section~\ref{section:preliminaries} we will see that our definition is equivalent to the categorical definition given by Grandis~\cite{grandis:shape}.
\begin{definition} \label{def:futureRetract}
A \emph{future retract} of $\fc(X)$ is a
subspace $A \subseteq X$ together with a homotopy class of dipaths
$[\gamma_x]$ for all $x\in X$, with $\gamma_x(0) = x$ and $\gamma_x(1)
=: x^+ \in A$ such that for all homotopy classes of dipaths
$[\gamma]:x \to a$ where $a \in A$, there is a unique morphism making
the following diagram commute.
\begin{equation} \label{eq:futureRetract}
\psset{arrows=->,nodesep=3pt,rowsep=1cm,colsep=1cm}
\begin{psmatrix}
  & a\\
x & x^+
\everypsbox{\scriptstyle}
\ncLine{2,1}{2,2}\aput(0.6){[\gamma_x]}
\ncLine{2,1}{1,2}\Aput{[\gamma]}
\ncLine[linestyle=dashed]{2,2}{1,2}
\end{psmatrix}
\end{equation}
We also insist that for $a \in A$, $[\gamma_a] = [\Id_a]$.
\end{definition}

\begin{example}
  In this example we describe a future retract of the square annulus,
  a subspace of $\dI \times \dI$. For all the points $x$ in the lower
  left square, $x^+ = a$ and for the remaining points $y$, $y^+ =
  b$. So $A = \{a,b\}$. We can think of the future retract as pushing
  points forward in time in a way so that no decisions are made with
  respect to the future.
\begin{center}
\psset{unit=2.5cm}
\begin{pspicture}(1,1)
\psset{fillstyle=solid,fillcolor=gray1,linecolor=black}
\psframe(0,0)(1,1)
\psset{fillcolor=white,linecolor=black}
\psframe(0.3333,0.3333)(0.6666,0.6666)
\psset{linecolor=black,fillstyle=none}
\psline{->}(0.3,0)(0.3001,0)
\psline{->}(0,0.55)(0,0.5501)
\psdots(0.3333,0.3333)(1,1)(0.1,0.2)(0.6,0.1)
\uput[ur](0.3333,0.3333){$a$}
\uput[r](1,1){$b$}
\uput[d](0.1,0.2){$x$}
\uput[l](0.6,0.1){$y$}
\psset{linecolor=blue,linestyle=dashed}
\psline{->}(0.1,0.2)(0.2166,0.2666)
\psline(0.2166,0.2666)(0.3333,0.3333)
\psline{->}(0.6,0.1)(0.8,0.55)
\psline(0.8,0.55)(1,1)
\psset{linecolor=yellow,linestyle=dotted}
\psline(0.3333,0)(0.3333,0.3333)(0,0.3333)
\end{pspicture}
\end{center}
\end{example}

We should not be unduly concerned that these retracts are not induced
by continuous maps. For in the classical undirected case, the skeleton
functor from the fundamental groupoid of $S^1$ to $\pi_1(S^1)$ is not
induced by a continuous map.

We remark that the definition implies that there is a unique morphism
making the following diagram commute.
\bigskip
\[
\begin{psmatrix}[arrows=->,nodesep=3pt,rowsep=1cm,colsep=1cm]
x & y\\
x^+  & y^+
\everypsbox{\scriptstyle}
\ncLine{1,1}{1,2}^{[\gamma]}
\ncLine{1,1}{2,1}\bput(0.4){[\gamma_x]}
\ncLine{1,2}{2,2}\aput(0.4){[\gamma_y]}
\ncLine[linestyle=dashed]{2,1}{2,2}^{[\gamma]^+}
\end{psmatrix}
\]
By uniqueness, $[\Id_x]^+ = [\Id_{x^+}]$ and $[\beta \circ \alpha]^+ =
[\beta]^+ \circ [\alpha]^+$. That is, we have a functor $P^+: \fc(X)
\to \fc(X,A)$.
Also note that $P^+ \fc(\iota) = \Id_{\fc(X,A)}$.

Dually, one has \emph{past retracts}, which induce a functor $P^-: \fc(X) \to \fc(X,A)$. For an explicit definition, see Definition~\ref{def:pastRetract}.

\subsection{Extremal models}

Just as we took the transitive symmetric closure of
\mbox{d--homo}topies, we are led to consider chains of past and future
retracts. In Definition~\ref{def:futurePastRetract}, we will generalize our previous definitions of future retracts and past retracts to full
subcategories of the fundamental category. This allows us to define the following new model of a \mbox{d--space} $X$.

\begin{definition} \label{def:extremalModel}
An
\emph{extremal model} of $X$ is a chain of
future retracts and  past retracts
\begin{equation} \label{eq:extremalModel}
  \fc(X) \xto{P^+_1} \fc(X,X_1) \xto{P^-_2}  \fc(X,X_2) \xto{P^+_3} \ldots \xto{P^{\pm}_n} \fc(X,A),
\end{equation}
such that $\Ext(X) \subseteq A$.
Call an extremal model \emph{minimal} if there are no nontrivial
future or past retracts $\fc(X,A) \to \fc(X,A')$ such that $\Ext(X) \subseteq A'$.
\end{definition}

\begin{example}
  Let $X$ be a nonempty path--connected topological space. Let $dX$ be
  the set of all paths in $X$ and choose $x \in X$.  Then $\fc(X)$ is
  the fundamental groupoid, and $\fc(X,x)$ is the fundamental
  group.  If $X = \{x\}$ then $\Ext(X) = \{x\}$, but otherwise
  $\Ext(X)$ is empty.  Set $[\gamma_x] = [\Id_x]$ and for all other $y
  \in X$ choose a homotopy class $[\gamma_y]$ of paths from $y$ to
  $x$. This induces a functor $\fc(X) \to \fc(X,x)$ which is the
  skeleton functor, a future retract, a past retract, and a minimal
  extremal model.
\end{example}

\begin{example}
Here we give three examples of an extremal model obtained from a future retract followed by a past retract. In each case, we have included the generating non-identity morphisms in the final figure.
\begin{enumerate}
\item
The square annulus:
\begin{center}
\psset{unit=2.5cm}
\begin{pspicture}(1,1)
\psset{fillstyle=solid,fillcolor=gray1,linecolor=black}
\psframe(0,0)(1,1)
\psset{fillcolor=white,linecolor=black}
\psframe(0.3333,0.3333)(0.6666,0.6666)
\psset{linecolor=black,fillstyle=none}
\psline{->}(0.3,0)(0.3001,0)
\psline{->}(0,0.55)(0,0.5501)
\psdots(0,0)(1,1)
\uput[l](0,0){$a$}
\uput[r](1,1){$b$}
\psset{linecolor=yellow,linestyle=dotted}
\psline(0.3333,0)(0.3333,0.3333)(0,0.3333)
\psset{linecolor=blue,linestyle=solid,doubleline=true,doublecolor=gray1,arrowscale =1.5}
\psline{->}(0.7777,0.7777)(0.8888,0.8888)
\end{pspicture}%
\quad \quad \quad \quad
\begin{pspicture}(1,1)
\psset{fillstyle=solid,fillcolor=gray1,linecolor=black}
\psframe(0,0)(0.3333,0.3333)
\psdots(0,0)(1,1)
\uput[l](0,0){$a$}
\uput[r](1,1){$b$}
\psset{linecolor=blue,linestyle=solid,doubleline=true,doublecolor=gray1,arrowscale =1.5}
\psline{->}(0.2222,0.2222)(0.1111,0.1111)
\end{pspicture}%
\quad \quad \quad \quad
\begin{pspicture}(1,1)
\psdots(0,0)(1,1)
\uput[l](0,0){$a$}
\uput[r](1,1){$b$}
\psset{linestyle=dashed,linecolor=red}
\pscurve(0,0)(0.6,0.4)(1,1)
\pscurve(0,0)(0.4,0.6)(1,1)
\psline{->}(0.4,0.6)(0.401,0.601)
\psline{->}(0.6,0.4)(0.601,0.401)
\end{pspicture}%
\end{center}
\item
The Swiss flag:
\begin{center}
\psset{unit=2.5cm}
\begin{pspicture}(1,1)
\psframe[linecolor=black,fillstyle=solid,fillcolor=gray1](0,0)(1,1)
\psline{->}(0.53,0)(0.5301,0)
\psline{->}(0,0.53)(0,0.5301)
\pspolygon[linecolor=black,fillstyle=solid,fillcolor=white](0.4,0.4)(0.2,0.4)(0.2,0.6)(0.4,0.6)(0.4,0.8)(0.6,0.8)(0.6,0.6)(0.8,0.6)(0.8,0.4)(0.6,0.4)(0.6,0.2)(0.4,0.2)
\psdots(0,0)(0.4,0.4)(0.6,0.6)(1,1)
\uput[l](0,0){$a$}
\uput[r](0.4,0.4){$b$}
\uput[l](0.6,0.6){$c$}
\uput[r](1,1){$d$}
\psset{linecolor=yellow,linestyle=dotted}
\psline(0.4,0)(0.4,0.2)(0.2,0.2)(0.2,0.4)(0,0.4)
\psline(0.8,0.6)(0.8,0.8)(0.6,0.8)
\psset{linecolor=blue,linestyle=solid,doubleline=true,doublecolor=gray1,arrowscale =1.5}
\psline{->}(0.25,0.25)(0.35,0.35)
\psline{->}(0.85,0.85)(0.95,0.95)
\end{pspicture}%
\quad \quad \quad \quad
\begin{pspicture}(1,1)
\pspolygon[linecolor=black,fillstyle=solid,fillcolor=gray1](0,0)(0.4,0)(0.4,0.2)(0.2,0.2)(0.2,0.4)(0,0.4)
\pspolygon[linecolor=black,fillstyle=solid,fillcolor=gray1](0.6,0.6)(0.6,0.8)(0.8,0.8)(0.8,0.6)
\psdots(0,0)(0.4,0.4)(0.6,0.6)(1,1)
\uput[l](0,0){$a$}
\uput[r](0.4,0.4){$b$}
\uput[l](0.6,0.6){$c$}
\uput[r](1,1){$d$}
\psset{linecolor=blue,linestyle=solid,doubleline=true,doublecolor=gray1,arrowscale =1.5}
\psline{->}(0.15,0.15)(0.05,0.05)
\psline{->}(0.75,0.75)(0.65,0.65)
\end{pspicture}%
\quad \quad \quad \quad
\begin{pspicture}(1,1)
\psdots(0,0)(0.4,0.4)(0.6,0.6)(1,1)
\uput[l](0,0){$a$}
\uput[r](0.4,0.4){$b$}
\uput[l](0.6,0.6){$c$}
\uput[r](1,1){$d$}
\psset{linestyle=dashed,linecolor=red}
\psline(0,0)(0.4,0.4)
\psline(0.6,0.6)(1,1)
\pscurve(0,0)(0.7,0.3)(1,1)
\pscurve(0,0)(0.3,0.7)(1,1)
\psline{->}(0.3,0.7)(0.301,0.701)
\psline{->}(0.7,0.3)(0.701,0.301)
\psline{->}(0.2,0.2)(0.201,0.201)
\psline{->}(0.8,0.8)(0.801,0.801)
\end{pspicture}%
\end{center}
\item
The directed square with two holes in series:
\begin{center}
\psset{unit=2.5cm}
\begin{pspicture}(1,1)
\psset{fillstyle=solid,fillcolor=gray1,linecolor=black}
\psframe(0,0)(1,1)
\psset{fillcolor=white,linecolor=black}
\psframe(0.2,0.2)(0.4,0.4)
\psframe(0.6,0.6)(0.8,0.8)
\psset{linecolor=black,fillstyle=none}
\psline{->}(0.3,0)(0.3001,0)
\psline{->}(0,0.55)(0,0.5501)
\psdots(0,0)(0.6,0.6)(1,1)
\uput[l](0,0){$a$}
\uput[ur](0.6,0.6){$b$}
\uput[r](1,1){$c$}
\psset{linecolor=yellow,linestyle=dotted}
\psline(0.6,0)(0.6,0.4)(0.4,0.4)(0.4,0.6)(0,0.6)
\psline(0.6,0.4)(0.6,0.6)(0.4,0.6)
\psset{linecolor=blue,linestyle=solid,doubleline=true,doublecolor=gray1,arrowscale =1.5}
\psline{->}(0.45,0.45)(0.55,0.55)
\psline{->}(0.85,0.85)(0.95,0.95)
\end{pspicture}%
\quad \quad \quad \quad
\begin{pspicture}(1,1)
\psset{fillstyle=solid,fillcolor=gray1,linecolor=black}
\pspolygon(0,0)(0.6,0)(0.6,0.4)(0.4,0.4)(0.4,0.6)(0,0.6)
\psset{fillcolor=white,linecolor=black}
\psframe(0.2,0.2)(0.4,0.4)
\psdots(0,0)(0.6,0.6)(1,1)
\uput[l](0,0){$a$}
\uput[ur](0.6,0.6){$b$}
\uput[r](1,1){$c$}
\psset{linecolor=blue,linestyle=solid,doubleline=true,doublecolor=gray1,arrowscale =1.5}
\psline{->}(0.15,0.15)(0.05,0.05)
\end{pspicture}%
\quad \quad \quad \quad
\begin{pspicture}(1,1)
\psdots(0,0)(0.6,0.6)(1,1)
\uput[l](0,0){$a$}
\uput[ur](0.6,0.6){$b$}
\uput[r](1,1){$c$}
\psset{linestyle=dashed,linecolor=red}
\pscurve(0,0)(0.2,0.4)(0.6,0.6)
\pscurve(0,0)(0.4,0.2)(0.6,0.6)
\pscurve(0.6,0.6)(0.7,0.9)(1,1)
\pscurve(0.6,0.6)(0.9,0.7)(1,1)
\psline{->}(0.2,0.4)(0.201,0.401)
\psline{->}(0.4,0.2)(0.401,0.201)
\psline{->}(0.7,0.9)(0.701,0.901)
\psline{->}(0.9,0.7)(0.901,0.701)
\end{pspicture}%
\end{center}
\end{enumerate}
In all three cases we obtain a minimal extremal model. The first two
are in fact equal to the fundamental bipartite graph. Notice that in the third example, we also have the branching information which is lost in the fundamental bipartite graph.
\end{example}

\begin{example}
  Let $x \in \dS$. The category $\fc(\dS, x)$ is isomorphic to the commutative monoid of non-negative integers under addition. For $y \in \dS$, let $[\gamma_y]$ be the homotopy class of dipaths from $y$ to $x$ such that no proper subpath of $\gamma_y$ is a dipath from $y$ to $x$. This defines a future retract of $\fc(\dS)$. The induced functor 
\begin{equation*}
P^+: \fc(\dS) \to \fc(\dS,x) \isom (\N,+)
\end{equation*}
is  a minimal extremal model.
\end{example}

The simple proof of the following is in Section~\ref{sec:fbg}.

\begin{proposition}  \label{prop:extremal-model}
  An extremal model induces an injection of fundamental bipartite graphs.
\end{proposition}

We will see that if a \mbox{d--space} $X$ is a compact pospace, then
this map is in fact an isomorphism (Theorem~\ref{prop:compactPospace}).

\subsection{Directed van Kampen theorems}

One of the main  tools for calculating the fundamental group and the
fundamental groupoid is the Seifert--van Kampen Theorem. A version of
this theorem also applies to the fundamental category. It was proved
by Goubault for local pospaces~\cite{goubault:someGPiCT} and by
Grandis for \mbox{d--spaces}~\cite{grandis:dht1}.  These proofs follow R. Brown's proof of the usual van Kampen theorem for
groupoids~\cite{brown:groupoidsAndVanKampen,brown:book}.

Let $X_1, X_2 \subseteq X$ be \mbox{d--spaces} with $X$ equal to the union of
the interiors of $X_1$ and $X_2$. Let $X_0 = X_1 \cap X_2$. Then
\[
\psset{arrows=->,nodesep=3pt,rowsep=1cm,colsep=1cm}
\begin{psmatrix}
X_0 & X_2\\
X_1  & X
\everypsbox{\scriptstyle}
\ncLine{1,1}{2,1}\Bput{i_1}
\ncLine{1,1}{1,2}^{i_2}
\ncLine{2,1}{2,2}^{j_1}
\ncLine{1,2}{2,2}\Aput{j_2}
\end{psmatrix}
\]
is a pushout in the category of \mbox{d--spaces}.
\begin{theorem}[\cite{grandis:dht1}]
 The induced commutative diagram 

\[
\psset{arrows=->,nodesep=3pt,rowsep=1cm,colsep=1cm}
\begin{psmatrix}
\fc(X_0) & \fc(X_2)\\
\fc(X_1) & \fc(X)
\everypsbox{\scriptstyle}
\ncLine{1,1}{2,1}\Bput{\fc(i_1)}
\ncLine{1,1}{1,2}\Aput{\fc(i_2)}
\ncLine{2,1}{2,2}\Aput{\fc(j_1)}
\ncLine{1,2}{2,2}\Aput{\fc(j_2)}
\end{psmatrix}
\]
is a pushout in the category of small categories.
\end{theorem}

We prove one version of this theorem for full (co)reflective subcategories, and another for future retracts and past retracts.
Let $X_1, X_2 \subseteq X$ be \mbox{d--spaces} with $X$ equal to the union of the interiors of $X_1$ and $X_2$, and $X_0 = X_1 \cap X_2$.  Let $A_1, A_2 \subseteq A$ be
\mbox{d--spaces} with $A = \Int(A_1) \cup \Int(A_2)$ and $A_0 = A_1 \cap
A_2$.
Assume that for $k=1,2,3$, $A_k \subseteq X_k$. So, we have the following commutative diagram of \mbox{d--spaces}.
\begin{equation} \label{cd:subspaces}
\psset{arrows=->,nodesep=3pt,rowsep=0.25cm,colsep=1cm}
\begin{psmatrix}
& A_0 & & A_2\\
A_1 & & A\\
& X_0 & & X_2\\
X_1 & & X
\everypsbox{\scriptstyle}
\ncLine{1,2}{1,4}
\ncLine{1,2}{2,1}
\ncLine{1,2}{3,2}
\ncLine{1,4}{2,3}
\ncLine{1,4}{3,4}
\ncLine{2,1}{2,3}
\ncLine{2,1}{4,1}
\ncLine{2,3}{4,3}
\ncLine{3,2}{3,4}
\ncLine{3,2}{4,1}
\ncLine{3,4}{4,3}
\ncLine{4,1}{4,3}
\end{psmatrix}
\end{equation}

\begin{theorem} \label{thm:vanKampen}
  Given compatible future retracts (solid arrows)
  \[
  \psset{arrows=->,nodesep=3pt,rowsep=0.5cm,colsep=1cm}
  \begin{psmatrix}
    & \fc(X_0,A_0) & & \fc(X_2,A_2)\\
    \fc(X_1,A_1) & & \fc(X,A)\\
    & \fc(X_0) & & \fc(X_2)\\
    \fc(X_1) & & \fc(X) \psset{shortput=nab} \everypsbox{\scriptstyle}
    \ncline{1,2}{1,4} \ncline{1,2}{2,1} \ncline{3,2}{1,2}
    \ncline{3,4}{1,4} \ncline{4,1}{2,1} \ncline{3,2}{3,4}
    \ncline{3,2}{4,1} \ncline[linestyle=dashed]{2,1}{2,3}
    \ncline[linestyle=dashed]{1,4}{2,3}
    \ncline[linestyle=dashed]{4,1}{4,3}
    \ncline[linestyle=dashed]{3,4}{4,3}
    \ncline[linestyle=dotted]{4,3}{2,3}
  \end{psmatrix}
  \]
  the top square, induced by \eqref{cd:subspaces}, is a pushout of categories, and there is an induced retraction (dotted arrow) on the pushouts, which makes the diagram commute.
\end{theorem}
  The dual statement holds for past retracts.

\medskip

We prove a more general version of Theorem~\ref{thm:vanKampen} (Theorem~\ref{thm:vanKampenABX}), for triples $A \subseteq B \subseteq X$.
This allows us to apply the theorem inductively to obtain an analogous theorem for chains of compatible future retracts and past retracts (Theorem~\ref{thm:vanKampenChain}). We use this to obtain a van Kampen theorem for extremal models (Theorem~\ref{thm:vanKampenExtremalModels}).
A simple application is given in Example~\ref{example:vanKampen}.

\section{Preliminaries} \label{section:preliminaries}

\subsection{Directed spaces and topological spaces}
\label{sec:directedSpaces}

We start by briefly relating directed spaces to topological spaces. Let $\Top$ and  $\dTop$ denote the categories of topological spaces, and \mbox{d--spaces}, respectively.

\begin{lemma}
  The underlying functor $U: \dTop \to \Top$, given by $U(X,dX) = X$ and
  $U(f) = f$ has a left adjoint $F$ given by the constant paths. That
  is, $F(X) = (X,dX)$ where $dX$ is the set of constant paths and
  $F(f) = f$.  The functor $U$ also has a right adjoint,
  $G$, given by all paths. That is, $G(X) = (X,dX)$
  where $dX$ is the set of all (ordinary) paths in $X$, and $G(f)=f$.
\end{lemma}

\begin{proof}
The following are natural isomorphisms:
\begin{gather*}
  \dTop(FX, (Y,dY)) \isom \Top(X,Y) \text{, and}\\
  \Top(X,Y) \isom \dTop((X,dX),GY). \qedhere
\end{gather*}
\end{proof}

\subsection{Future retracts and past retracts}

For the convenience of the reader, we define past retracts explicitly.

\begin{definition} \label{def:pastRetract} 
A \emph{past retract} of $\fc(X)$   is a subspace $A \subseteq X$ together with a directed homotopy class of   dipaths $[\gamma_x]$ with $\gamma_x(1) = x$ and $\gamma_x(0) =: x^- \in A$   such that for any $[\gamma]:a \to x$ with $a \in A$, there is a unique   morphism making the following diagram commute.
\[
\psset{arrows=->,nodesep=3pt,rowsep=1cm,colsep=1cm}
\begin{psmatrix}
x^- & x\\
a
\everypsbox{\scriptstyle}
\ncLine{1,1}{1,2}^{[\gamma_x]}
\ncLine{2,1}{1,2}\Bput{[\gamma]}
\ncLine[linestyle=dashed]{2,1}{1,1}
\end{psmatrix}
\]
Again, we also insist that for $a \in A$, $[\gamma_a] = [\Id_a]$.
We obtain a functor $P^-: \fc(X) \to \fc(X,A)$, with $P^-(\fc(\iota)) = \Id_{\fc(X,A)}$. 
\end{definition}

\begin{example}
The following past retract is dual to the previous example of a future retract.
\begin{center}
\psset{unit=2.5cm}
\begin{pspicture}(1,1)
\psset{fillstyle=solid,fillcolor=gray1,linecolor=black}
\psframe(0,0)(1,1)
\psset{fillcolor=white,linecolor=black}
\psframe(0.3333,0.3333)(0.6666,0.6666)
\psset{linecolor=black,fillstyle=none}
\psline{->}(0.3,0)(0.3001,0)
\psline{->}(0,0.55)(0,0.5501)
\psdots(0,0)(0.6666,0.6666)
\uput[l](0,0){$a$}
\uput[dl](0.6666,0.6666){$b$}
\psset{linecolor=yellow,linestyle=dotted}
\psline(0.6666,1)(0.6666,0.6666)(1,0.6666)
\psset{linecolor=blue,linestyle=solid,doubleline=true,doublecolor=gray1,arrowscale =1.5}
\psline{->}(0.2222,0.2222)(0.1111,0.1111)
\psline{->}(0.8888,0.8888)(0.7777,0.7777)
\end{pspicture}
\end{center}
\end{example}

We now show that future retracts and past retracts have a succinct
categorical definition. In fact, this is how they were first defined
by Grandis~\cite{grandis:shape} (who defined them for the fundamental category of a preordered space). Recall that a future retract induces
a functor $P^+: \fc(X) \to \fc(X,A)$ with $P^+(\fc(\iota)) =
\Id_{\fc(X,A)}$, and that a past retract induces a functor $P^-: \fc
\to \fc(X,A)$ with $P^-(\fc(\iota)) = \Id_{\fc(X,A)}$.

\begin{proposition}
  There is a bijection between future retracts $\iota:A \subseteq X$ and   adjunctions
  \begin{equation*}
    P^+: \fc(X) \adjn \fc(X,A): \fc(\iota)
  \end{equation*}
with $P^+(\fc(\iota)) = \Id_{\fc(X,A)}$.
\end{proposition}

\begin{proof}
  $(\Rightarrow)$ We've already shown that a future retract defines a
  functor $P^+: \fc(X) \to \fc(X,A)$ with $P^+(\fc(\iota)) =
  \Id_{\fc(X,A)}$. The assignment $\eta_x: x \xto{[\gamma_x]} x^+$ is
  universal from $x$ to $\fc(\iota)$, and determines a natural
  transformation $\eta: \Id_{\fc(X)} \to \fc(\iota)P^+$. Therefore
  $P^+$ is the left adjoint of $\fc(\iota)$.

  $(\Leftarrow)$ Assume we are given an adjunction $P^+: \fc(X) \adjn
  \fc(X,A): \fc(\iota)$. For $x \in X$, the unit $\eta_x:x \to x^+$ is
  universal from $x$ to $\fc(\iota)$. That is, there is a unique
  morphism making the following diagram commute.
\[
\psset{arrows=->,nodesep=3pt,rowsep=1cm,colsep=1cm}
\begin{psmatrix}
x & x^+\\
  & a
\everypsbox{\scriptstyle}
\ncLine{1,1}{1,2}^{\eta_x}
\ncLine[linestyle=dashed]{1,2}{2,2}_{[\gamma]}
\ncLine{1,1}{2,2}
\end{psmatrix} \qedhere
\]
\end{proof}

Since the inclusion of the full subcategory $\fc(X,A)$ has a left adjoint, we say that $\fc(X,A)$ is a \emph{full reflective subcategory} of $\fc(X)$.
Dually, past retracts are equivalent to \emph{full coreflective subcategories}.

\begin{proposition}
  There is a bijection between past retracts $\iota:A \subseteq X$ and
  adjunctions
  \begin{equation*}
    \fc(\iota): \fc(X,A) \adjn \fc(X): P^-
  \end{equation*}
  with $P^-(\fc(\iota)) = \Id_{\fc(X,A)}$.
\end{proposition}

\begin{remark}
  It follows that for future retracts and past retracts we have the following
  natural isomorphisms. For $x\in X$ and $a \in A$,
  \begin{gather} \label{eq:retractIsos}
    \fc(X,A)(x^+,a) \isom \fc(X)(x,a)\\
    \fc(X)(a,x) \isom \fc(X,A)(a,x^-)
  \end{gather}
\end{remark}

Generalizing Definitions \ref{def:futureRetract} and \ref{def:pastRetract} in the present language, we have:

\begin{definition} \label{def:futurePastRetract}
  A future (past) retract of $\fc(X,A)$ is a full (co)reflective subcategory $\fc(X,B)$, with $P^{\pm}(\fc(\iota)) = \Id_{\fc(X,A)}$.
\end{definition}

\section{The fundamental bipartite graph}
\label{sec:fbg}

Let $X$ have an extremal model: a chain of future retracts and past
retracts
\begin{equation*} 
  \fc(X) \xto{P^+_1} \fc(X,X_1) \xto{P^-_2}  \fc(X,X_2) \xto{P^+_3} \ldots \xto{P^{\pm}_n} \fc(X,A),
\end{equation*}
such that $\Ext(X) \subseteq A$.

\begin{proposition} \label{prop:injection}
  An extremal model induces an injection of fundamental bipartite graphs.
\end{proposition}

\begin{proof}
  By definition, $\Ext(\fc(X)) \subseteq A$. For $a \in A$, since
  $\fc(X,A)$ is a subcategory of $\fc(X)$, if $a \notin
  \Ext(\fc(X,A))$ then $a \notin \Ext(\fc(X))$. Combining these two
  facts we obtain that 
  \begin{equation*}
    \Ext(\fc(X)) \subseteq \Ext(\fc(X,A)).
  \end{equation*}
  Thus
  $\fc(X,\Ext(\fc(X)))$ is a subcategory of $\fc(X,\Ext(\fc(X,A)))$.
\end{proof}

The map induced by future retracts and past retracts on the
fundamental bipartite graph is not surjective in general. For example,
take the unit interval $[0,1]$ and all (undirected) paths, and let $x
\in [0,1]$. Then the map $[0,1] \to \{x\}$ induces a past and future
retract. However $\Ext([0,1])$ is empty while $\Ext(\{x\}) = \{x\}$.

We will show that if a \mbox{d--space} $X$ is a compact pospace, then
this map is in fact an isomorphism.

\begin{definition}
  A \emph{pospace} is a topological spaces $X$, together with a reflexive, transitive, anti-symmetric relation $\leq$, such that $\leq$ is a closed subset of $X \times X$ in the product topology.
\end{definition}

Given a \mbox{d--space} $X$, the fundamental category $\fc(X)$ induces
a preorder on $X$. Assume that this order makes $X$ into a compact
pospace.  Let
\begin{equation} \label{eq:compactModel}
  \fc(X) \xto{P^+_1} \fc(X,X_1) \xto{P^-_2}  \fc(X,X_2) \xto{P^+_3} \ldots \xto{P^{\pm}_n} \fc(X,A),
\end{equation}
be an extremal model of $X$ in which for $1\leq i \leq n$, $X_i$ is compact.

\begin{theorem} \label{prop:compactPospace}
  Such an extremal model of a compact pospace induces an isomorphism of
  fundamental bipartite graphs.
\end{theorem}

\begin{proof}
  Let $X$ be as above. Our proof is by induction on the number of
  retracts in the extremal model.

  Let $P: \fc(X) \to \fc(X,B)$ be an extremal model as in
  \eqref{eq:compactModel} and let $P^+: \fc(X,B) \to \fc(X,A)$ be a
  future retract. By Proposition~\ref{prop:injection}, $P^+ \circ P$
  is injective on extremal points. We will show that $P^+:
  \Ext(\fc(X,B)) \onto \Ext(\fc(X,A))$. It will follow by induction
  that $P^+ \circ P: \Ext(\fc(X)) \isomto \Ext(\fc(X,A))$.

  Let $a$ be a maximal point in $\fc(X,A)$. Let $b \in B$, with $a
  \leq b$. Then $a \leq b \leq b^+$. Since $b^+ \in A$ and $a$ is
  maximal in $\fc(X,A)$, $a=b^+$. Since $\leq$ is anti-symmetric, it
  follows that $a=b$. Therefore, $a$ is maximal in $\fc(X)$. Thus the
  maximal points in $\fc(X,A)$ are also maximal in $\fc(X,B)$.

  Let $a$ be a minimal point in $\fc(X,A)$.  Since $A \subseteq B$, $a
  \in B$.  By assumption $\fc(X)$ induces an order $\leq$ on $X$ such
  that $X$ is a pospace. Order $B$ with the order induced by
  $\leq$. This coincides with the order induced by $\fc(X,B)$. It is
  well--known and easy to check that the induced order on a subspace
  of a pospace gives it the structure of a pospace. By assumption, $B$
  is compact.  Since $B$ is a compact pospace, there is a minimal
  point $b \in \fc(X,B)$ such that $b \leq
  a$ \cite{wallace:aFixedPointTheorem} \cite[Proposition VI-5.3]{continuousLatticesAndDomains}. Since $P^+$ is a future
  retract, $b^+ \leq a$. Since $a$ is minimal in $\fc(X,A)$ and $b^+
  \in A$ it follows that $a = b^+$.
\end{proof}

We remark that the compact condition is necessary. Consider $\R$ with
$d\R$ all nondecreasing paths $[0,1] \to \R$. Then the induced order
is the usual total order on $\R$ and it makes $\R$ into a pospace. There is
a future retract $P^+$ from $\R$ to the non-negative real numbers $\R_{\geq 0}$, where $x^+ = x$ if $x \geq 0$ and $x^+ = 0$ if $x<0$. However $\Ext(\R)$ is empty while $\Ext(\R_{\geq 0}) = \{0\}$.

\section{Directed van Kampen theorems}
\label{sec:directed-van-kampen}

We start this section by proving a version of the Seifert -- van
Kampen Theorem for full subcategories of the fundamental category
(Theorem~\ref{thm:vanKampen}). Our proof follows Grandis' proof of the
van Kampen Theorem for \mbox{d--spaces}~\cite{grandis:dht1}, which in turn
follows R. Brown's proof of the usual van Kampen Theorem for
groupoids~\cite{brown:groupoidsAndVanKampen,brown:book}.  Instead of
working with $A\subseteq X$ and the full subcategory $\fc(X,A)$ of
$\fc(X)$, we work in the more general setting $A \subseteq B \subseteq
X$ and the full subcategory $\fc(X,A) \subseteq \fc(X,B)$. The former
can be obtained from the latter by setting $B=X$.

Next we prove a van Kampen Theorem for past and future retracts. This construction is shown to preserve the non-collapsing property. 
Finally we prove a van Kampen Theorem for chains of past retracts and future retracts. As a corollary, we obtain a van Kampen theorem for extremal models. 

Let $X_1, X_2 \subseteq X$ be \mbox{d--spaces} with $X$ equal to the union of the interiors of $X_1$ and $X_2$.  Let $X_0 = X_1 \cap X_2$. With these statements we assume
that the \mbox{d--space} structure on $X$ is induced by the \mbox{d--space} structures on
$X_1$ and $X_2$. That is, 
dipaths in $X$ are concatenations of dipaths in $X_1$ and $X_2$.

  Similarly let $A_1, A_2 \subseteq A$ be
\mbox{d--spaces} with $A = \Int(A_1) \cup \Int(A_2)$ and let $B_1, B_2 \subseteq B$ be
\mbox{d--spaces} with $B = \Int(B_1) \cup \Int(B_2)$ . Let $A_0 = A_1 \cap
A_2$ and $B_0 = B_1 \cap
B_2$.
Assume that for $k=1,2,3$, $A_k \subseteq B_k \subseteq X_k$. Thus we have the following commutative diagram of \mbox{d--spaces}.

\begin{equation} \label{cd:ABX}
\psset{arrows=->,nodesep=3pt,rowsep=0.5cm,colsep=1cm}
\begin{psmatrix}
& A_0 & & A_2\\
A_1 & & A\\
& B_0 & & B_2\\
B_1 & & B\\
& X_0 & & X_2\\
X_1 & & X
\everypsbox{\scriptstyle}
\ncLine{1,2}{1,4}\Aput{i'_2}
\ncLine{1,2}{2,1}\Bput{i'_1}
\ncLine{1,4}{2,3}\bput(0.7){j'_2}
\ncLine{2,1}{2,3}\aput(0.75){j'_1}
\ncLine{1,2}{3,2}\bput(0.75){{\iota}_0}
\ncLine{2,1}{4,1}\Bput{\iota_1}
\ncLine{1,4}{3,4}\Aput{\iota_2}
\ncLine{2,3}{4,3}\bput(0.3){\iota}
\ncLine{3,2}{3,4}\aput(0.75){i_2}
\ncLine{3,2}{4,1}\Bput{i_1}
\ncLine{3,4}{4,3}\Aput{j_2}
\ncLine{4,1}{4,3}\aput(0.75){j_1}
\ncLine{4,1}{6,1}
\ncLine{3,2}{5,2}
\ncLine{3,4}{5,4}
\ncLine{4,3}{6,3}
\ncLine{5,2}{5,4}
\ncLine{5,2}{6,1}
\ncLine{5,4}{6,3}
\ncLine{6,1}{6,3}
\end{psmatrix}
\end{equation}

Furthermore, assume that $\fc(X_k,A_k) \subseteq \fc(X_k,B_k)$ is a full reflective subcategory for $k=0,1,2$ and that the following diagram commutes, where $P^+_k$ denotes the reflections.
  \begin{equation} \label{cd:Pk}
  \psset{arrows=->,nodesep=3pt,rowsep=0.25cm,colsep=1cm}
  \begin{psmatrix}
    & \fc(X_0,A_0) & & \fc(X_2,A_2)\\
    \fc(X_1,A_1)\\
    & \fc(X_0,B_0) & & \fc(X_2,B_2)\\
    \fc(X_1,B_1)
\everypsbox{\scriptstyle}
    \ncline{1,2}{1,4} \ncline{1,2}{2,1} \ncline{3,2}{1,2}>{P^+_0}
    \ncline{3,4}{1,4}>{P^+_2} \ncline{4,1}{2,1}<{P^+_1} \ncline{3,2}{3,4}
    \ncline{3,2}{4,1}
  \end{psmatrix}
  \end{equation}

The following is our main lemma. We assume \eqref{cd:ABX} and \eqref{cd:Pk} with $B_k = X_k$ for $k=0,1,2$.

\begin{lemma} \label{lemma:main}
  Let $[\gamma] \in \fc(X,A)$. Then there exist $\gamma_1, \ldots, \gamma_n$ with $[\gamma_i] \in \fc(X_1,A_i)$ or $\fc(X_2,A_2)$ such that $[\gamma] = [\gamma_1] + \ldots + [\gamma_2]$.
\end{lemma}

\begin{proof}
  Let $[\gamma] \in \fc(X,A)$ with $\gamma(0)=a$ and $\gamma(1)=a'$.
  By the Lebesgue number lemma, there is a number $n$ such that $\gamma\left(\left[\frac{i-1}{n},\frac{i}{n}\right]\right) \subseteq X_{k_i}$ where $k_i \in \{1,2\}$ for all $i=1,\ldots,n$.
  Let $x_i = \gamma(\frac{i}{n})$, $i=0,\ldots,n$.
  Let $\gamma_i:\dI \to X$ be given by $\gamma_i(t) = \gamma(\frac{i-1+t}{n})$. 
  Then $\gamma = \gamma_1 + \ldots + \gamma_n$, $\gamma_i: \dI \to X_{k_i}$ and $[\gamma_i] \in \fc(X_{k_i})$. The only remaining problem is that we do not have $[\gamma_i] \in \fc(X_{k_i},A_{k_i})$.

  Let $[\gamma_i]^+$ denote $P^+_{k_i}[\gamma_i]$. These maps of paths
  induce maps $x_i \mapsto x_i^+$ which are well-defined by the
  commutativity of~\eqref{cd:Pk}.  
Composing the commutative diagrams
\[
\psset{arrows=->,nodesep=3pt,rowsep=1cm,colsep=1cm}
\begin{psmatrix}
x_{i-1} & x_{i-1}^+\\
x_i  & x_i^+
\psset{shortput=nab}
\everypsbox{\scriptstyle}
\ncLine{1,1}{1,2}
\ncLine{1,1}{2,1}_{[\gamma_i]}
\ncLine{2,1}{2,2}
\ncLine{1,2}{2,2}^{[\gamma_i]^+}
\end{psmatrix}
\]
  we obtain
  \begin{equation*}
    [\gamma] = [\gamma_1] + \ldots + [\gamma_n] = [\gamma_1]^+ + \ldots + [\gamma_n]^+,
  \end{equation*} 
where $[\gamma_i]^+ \in \fc(X_{k_i},A_{k_i})$.
\end{proof}

\begin{theorem} \label{thm:pushoutAinX}
  The following diagram is a pushout of categories.

\[
\psset{arrows=->,nodesep=3pt,rowsep=1cm,colsep=2cm}
\begin{psmatrix}
\fc(X_0,A_0) & \fc(X_2,A_2)\\
\fc(X_1,A_1) & \fc(X,A)
\psset{shortput=nab}
\everypsbox{\scriptstyle}
\ncLine{1,1}{2,1}_{\fc(i'_1)}
\ncLine{1,1}{1,2}^{\fc(i'_2)}
\ncLine{2,1}{2,2}^{\fc(j'_1)}
\ncLine{1,2}{2,2}^{\fc(j'_2)}
\end{psmatrix}
\]
\end{theorem}


\begin{proof}
  Let $\C$ be a category.  Assume $\phi_k: \fc(X_k,A_k) \to \C$ for
  $k=1,2$ such that $\phi_1 \fc(i'_1) = \phi_2 \fc(i'_2)$.  Let
  $[\gamma] \in \fc(X,A)$ with $\gamma(0)=a$ and $\gamma(1)=a'$. Apply
  Lemma~\ref{lemma:main} to $[\gamma]$. Define $F[\gamma] =
  \phi_{k_1}[\gamma_1] + \ldots + \phi_{k_n}[\gamma_n]$, where
  addition is given by composition in $\C$.

We first remark that $F$ does not depend on the choice of $k_i$. If $\im(\gamma_i) \subset X_1 \cap X_2 = X_0$, then the compatibility of $\phi_1$ and $\phi_2$ ensures that $\phi_1[\gamma_i] = \phi_2[\gamma_i]$.

Next, $F$ does not depend on the choice of $n$: given another suitable $m$, consider the partition into $nm$ pieces.

Finally, $F$ does not depend on the choice of representative $\gamma$. Consider another $\tilde{\gamma} \simeq \gamma$.
Again, apply Lebesgue's number lemma to $I \times I$ to suitably decompose the homotopy from $\gamma$ to $\tilde{\gamma}$ into homotopies contained in either $X_1$ or $X_2$. Now apply the suitable choice of $P^+_1$ or $P^+_2$ to each of these. Use the resulting set of homotopies in $\fc(X_1,A_1)$ and $\fc(X_2,A_2)$ to obtain 
\begin{multline*}
  F[\gamma] = F[\gamma_1] + \ldots + F[\gamma_n] + F[\Id_{a'}]\\
 = F[\Id_a] + F[\tilde{\gamma}_1] + \ldots + F[\tilde{\gamma}_n] = F[\tilde{\gamma}].
\end{multline*}
Therefore $F$ is well defined.

For functoriality, notice that $F$ preserves compositions: if $\gamma,\gamma'$ have decompositions $\gamma=\gamma_1 + \ldots + \gamma_n$ and $\gamma' = \gamma'_1 + \ldots + \gamma'_m$, then $\gamma+\gamma'$ has decomposition $\gamma_1 + \ldots + \gamma_n + \gamma'_1 + \ldots + \gamma'_m$.

The uniqueness of $F$ is by construction.
\end{proof}

\begin{lemma} \label{lemma:F'}
  Given the following commutative solid--arrowed diagram. Let $F$ and $F'$ be the pushout maps.
  \[
  \psset{unit=0.5cm}
  \psset{arrows=->,nodesep=3pt,rowsep=0.5cm,colsep=1cm}
  \begin{psmatrix}
    & \fc(X_0,A_0) & & \fc(X_2,A_2)\\
    \fc(X_1,A_1) & & \fc(X,A)\\
    & \fc(X_0,B_0) & & \fc(X_2,B_2)\\
    \fc(X_1,B_1) & & \fc(X,B)\\
    & & & & & \C
\everypsbox{\scriptstyle}
    \ncline{1,2}{1,4} \ncline{1,2}{2,1} \ncline{1,2}{3,2}
    \ncline{1,4}{3,4}\bput(0.35){\fc(\iota_2)}
    \ncline{2,1}{4,1}\bput(0.4){\fc(\iota_1)} 
    \ncline{3,2}{3,4}
    \ncline{3,2}{4,1} \ncline{2,3}{4,3}\bput(0.15){\fc(\iota)} 
    \ncline{2,1}{2,3}
    \ncline{1,4}{2,3}
    \ncline{4,1}{4,3}
    \ncline{3,4}{4,3}
    \ncline{1,4}{5,6}^{\phi'_2} \ncline{3,4}{5,6}\aput(0.3){\phi_2} \ncline{2,1}{5,6}\bput(0.05){\phi'_1} \ncline {4,1}{5,6}_{\phi_1}
    \psset{linestyle=dashed}
    \ncline{2,3}{5,6}\aput(0.15){F'} \ncline{4,3}{5,6}\aput(0.1){F}
  \end{psmatrix}
  \]
  Then $F' = F \fc(\iota)$.
\end{lemma}

\begin{proof}
  Let $[\gamma] \in \fc(X,A)$. Apply Lemma~\ref{lemma:main} to $[\gamma]$.
  \begin{equation*}
    \begin{split}
      F'[\gamma] & = \phi'_{k_1}[\gamma_1]^+ + \ldots + \phi'_{k_n}[\gamma_n]^+\\
      & = \phi_{k_1} \fc(\iota_{k_1})[\gamma_1]^+ + \ldots + \phi_{k_n} \fc(\iota_{k_n})[\gamma_n]^+\\
      & = \phi_{k_1} [\gamma_1]^+ + \ldots + \phi_{k_n} [\gamma_n]^+\\
      & = F[\gamma]\\
      & = F \fc(\iota) [\gamma] \qedhere
    \end{split}
  \end{equation*}
\end{proof}

Let $\Cat$ denote the category of categories.
\begin{theorem} \label{thm:pushoutIota}
  The following diagram is a pushout in the arrow category on $\Cat$.
  \[
  \psset{arrows=->,nodesep=3pt,rowsep=0.5cm,colsep=1cm}
  \begin{psmatrix}
    & \fc(X_0,A_0) & & \fc(X_2,A_2)\\
    \fc(X_1,A_1) & & \fc(X,A)\\
    & \fc(X_0,B_0) & & \fc(X_2,B_2)\\
    \fc(X_1,B_1) & & \fc(X,B) \psset{shortput=nab} \everypsbox{\scriptstyle}
    \ncline{1,2}{1,4} \ncline{1,2}{2,1} \ncline{1,2}{3,2}
    \ncline{1,4}{3,4} \ncline{2,1}{4,1} \ncline{3,2}{3,4}
    \ncline{3,2}{4,1} \ncline{2,1}{2,3}
    \ncline{1,4}{2,3}
    \ncline{4,1}{4,3}
    \ncline{3,4}{4,3}
    \ncline{2,3}{4,3}
  \end{psmatrix}
  \]
\end{theorem}

\begin{proof}
  Let $F: \C \to \D$ be a functor between categories $\C$ and $\D$.
  We wish to show that given the solid--arrowed commutative diagram
  below, there are unique maps $G$ and $H$ making the diagram commute.
  \[
  \psset{arrows=->,nodesep=3pt,rowsep=0.5cm,colsep=1cm}
  \begin{psmatrix}
    & \fc(X_0,A_0) & & \fc(X_2,A_2)\\
    \fc(X_1,A_1) & & \fc(X,A)\\
    & \fc(X_0,B_0) & & \fc(X_2,B_2)\\
    \fc(X_1,B_1) & & \fc(X,B)\\
    & & & & & \C\\
    & & & & & \D
\everypsbox{\scriptstyle}
    \ncline{1,2}{1,4} \ncline{1,2}{2,1} \ncline{1,2}{3,2}
    \ncline{1,4}{3,4} \ncline{2,1}{4,1} \ncline{3,2}{3,4}
    \ncline{3,2}{4,1} \ncline{2,3}{4,3} 
    \ncline{2,1}{2,3}
    \ncline{1,4}{2,3}
    \ncline{4,1}{4,3}
    \ncline{3,4}{4,3}
    \ncline{5,6}{6,6}\Aput{F}
    \ncline{1,4}{5,6}>{\phi'_2} 
    \ncline{3,4}{6,6}
    \ncline{2,1}{5,6}<{\phi'_1} 
    \ncline {4,1}{6,6}
    \psset{linestyle=dashed}
    \ncline{2,3}{5,6}\aput(0.75){G} \ncline{4,3}{6,6}\Aput{H}
  \end{psmatrix}
  \]
  Since the top and bottom squares are pushouts, there are unique maps $G$ and $H$ making the top and bottom commute. For commutativity it remains to show that 
the the following diagram commutes.
\[
\psset{arrows=->,nodesep=3pt,rowsep=1cm,colsep=1cm}
\begin{psmatrix}
\fc(X,A) & \C\\
\fc(X,B) & \D
\psset{shortput=nab}
\everypsbox{\scriptstyle}
\ncLine{1,1}{1,2}^{G}
\ncLine{1,1}{2,1}_{\fc(\iota)}
\ncLine{2,1}{2,2}^H
\ncLine{1,2}{2,2}^F
\end{psmatrix}
\]
Since $FG$ is the pushout map of the following diagram
\[
\psset{arrows=->,nodesep=3pt,rowsep=1cm,colsep=1cm}
\begin{psmatrix}
\fc(X_0,A_0) & \fc(X_2,A_2)\\
\fc(X_1,A_1) & \fc(X,A)\\
& & \D
\everypsbox{\scriptstyle}
\ncLine{1,1}{2,1}
\ncLine{1,1}{1,2}
\ncLine{2,1}{2,2}
\ncLine{1,2}{2,2}
\ncLine{1,2}{3,3}\Aput{F\phi_2}
\ncLine{2,1}{3,3}_{F\phi_1}
\ncLine[linestyle=dashed]{2,2}{3,3}<{FG}
\end{psmatrix}
\]
Lemma~\ref{lemma:F'} tells us that $FG = H \fc(\iota)$.

Finally, non-uniqueness of $(G,H)$ would contradict the uniqueness of $G$ and $H$.
\end{proof}

Given the commutative diagram~\eqref{cd:Pk} recall that $\fc(X,B)$ and $\fc(X,A)$ are the pushouts of the bottom and the top respectively. We will define a functor $P^+:\fc(X,B) \to \fc(X,A)$ and show that it is the pushout in the arrow category of $\Cat$.
  \begin{equation} \label{cd:P+}
  \psset{arrows=->,nodesep=3pt,rowsep=1cm,colsep=1cm}
  \begin{psmatrix}
    & \fc(X_0,A_0) & & \fc(X_2,A_2)\\
    \fc(X_1,A_1) & & \fc(X,A)\\
    & \fc(X_0,B_0) & & \fc(X_2,B_2)\\
    \fc(X_1,B_1) & & \fc(X,B) 
    \everypsbox{\scriptstyle}
    \ncline{1,2}{1,4}\Aput{\fc(\iota'_2)} \ncline{1,2}{2,1}\Bput{\fc(\iota'_1)}
    \ncline{3,2}{1,2}
    \ncline{4,1}{2,1}<{P^+_1} \ncline{3,2}{3,4}\aput(0.25){\fc(i_2)}
    \ncline{3,2}{4,1}\Bput{\fc(i_1)} 
    \ncline{3,4}{1,4}<{P'_2}
    \psset{linestyle=dashed}
    \ncline{2,1}{2,3}\aput(0.75){\fc(j'_1)} \ncline{1,4}{2,3}\bput(0.7){\fc(j'_2)}
    \ncline{3,4}{4,3}\bput(0.6){\fc(j_2)}
    \ncline{4,1}{4,3}\Aput{\fc(j_1)}
    \psset{linestyle=dotted}  \ncLine{4,3}{2,3}\bput(0.7){P^+}
  \end{psmatrix}
\end{equation}

\begin{definition} \label{def:P+}
  Define $P^+: \fc(X,B) \to \fc(X,A)$ as follows. For $x\in \fc(X,B)$, 
  \begin{equation*}
    P^+: x \mapsto
    \begin{cases}
      \fc(j'_1) P^+_1 x & \text{if } x \in B_1\\
      \fc(j'_2) P^+_2 x & \text{if } x \in B_2\\
    \end{cases}
  \end{equation*}
  Is this well-defined? If $x \in B_1 \cap B_2 = B_0$, then they agree
  by the commutativity of the solid and dashed arrows
  in~\eqref{cd:P+}. Let $[\gamma] \in \fc(X,B)$. By Lemma~\ref{lemma:main} there exist $\gamma_1, \ldots, \gamma_n$ with $[\gamma_i] \in \fc(X_{k_i},B_{k_i})$ for $k_i \in \{1,2\}$ such that $[\gamma] = [\gamma_1] + \ldots +
  [\gamma_n]$. Define
\begin{equation*}
P^+[\gamma] = \fc(j'_{k_1}) P^+_{k_1} [\gamma_1] + \ldots + \fc(j'_{k_n}) P^+_{k_n} [\gamma_n].
\end{equation*}
This well defined by the same argument as in the proof of
Theorem~\ref{thm:pushoutAinX}.  It will be convenient to denote
$P^+(x)$ and $P^+[\gamma]$ by $x^+$ and $[\gamma]^+$ respectively.
\end{definition}

\begin{lemma} \label{lemma:unit}
  Let $[\gamma] \in \fc(X,B)$ with $\gamma(0)=x$, $\gamma(1)=y$.
Then the following diagram commutes.
\[
\psset{arrows=->,nodesep=3pt,rowsep=1cm,colsep=1cm}
\begin{psmatrix}
x & x^+\\
y  & y^+
\psset{shortput=nab}
\everypsbox{\scriptstyle}
\ncLine{1,1}{1,2}
\ncLine{1,1}{2,1}_{[\gamma]}
\ncLine{2,1}{2,2}
\ncLine{1,2}{2,2}^{[\gamma]^+}
\end{psmatrix}
\]
\end{lemma}

\begin{proof}
Let $\gamma_1, \ldots, \gamma_n$ be as in Definition~\ref{def:P+}. Let $[\gamma_i]^+$ denote $\fc(j'_{k_i})P^+_{k_i}[\gamma_i]$.
Composing the commutative diagrams
\[
\psset{arrows=->,nodesep=3pt,rowsep=1cm,colsep=1cm}
\begin{psmatrix}
x_{i-1} & x_{i-1}^+\\
x_i  & x_i^+
\psset{shortput=nab}
\everypsbox{\scriptstyle}
\ncLine{1,1}{1,2}
\ncLine{1,1}{2,1}_{[\gamma_i]}
\ncLine{2,1}{2,2}
\ncLine{1,2}{2,2}^{[\gamma_i]^+}
\end{psmatrix}
\]
  we obtain the desired result.
\end{proof}

\begin{theorem} \label{thm:vanKampenABX}
  In \eqref{cd:P+}, $P^+$ is a pushout in the arrow category on $\Cat$.
\end{theorem}

\begin{proof}
  The theorem follows by the same argument as the one used to prove Theorem~\ref{thm:pushoutIota}.
\end{proof}

\begin{theorem}
  Assume that $P^+_k$ is the left adjoint of $\fc(\iota_k)$ for $k=0,1,2$. Then there is an adjunction,
  \begin{equation*}
    P^+: \fc(X,B) \adjn \fc(X,A): \fc(\iota).
  \end{equation*}
\end{theorem}

\begin{proof}
  The unit $\eta^+: 1_{\fc(X,B)} \to \fc(\iota)P^+$ is a natural
  transformation by Lemma~\ref{lemma:unit}. The counit $\epsilon^+:
  P^+\fc(\iota) \to 1_{\fc(X,A)}$ is given by the identity.
Finally, $\epsilon^+_{x^+} \circ P^+(\eta^+_{x^+}) = \Id_{x^+}$ and $\fc(\iota)(\epsilon^+_a) \circ \eta^+_a = \Id_a$.
\end{proof}

Assume that for $k=0,1,2$ we have chains of future retracts and past retracts
\begin{multline*}
  P_k: \fc(X_{k,0}) \xto{P^+_{k,1}} \fc(X_{k,0},X_{k,1}) \xto{P^-_{k,2}} \\ \fc(X_{k,0},X_{k,2}) \xto{P^+{k,3}} \ldots \to \fc(X_{k,0},X_{k,n})
\end{multline*}
that are \emph{compatible}. That is, for $\ell=1,\ldots,n$, $X_{1,\ell} \cap X_{2,\ell} = X_{0,\ell}$, $X_{\ell} = \Int(X_{1,\ell}) \cup \Int(X_{2,\ell})$, and the diagrams corresponding to \eqref{cd:Pk}, but with $P_{0,\ell}$, $P_{1,\ell}$ and $P_{2,\ell}$, commute. 
Apply Theorem~\ref{thm:pushoutAinX} to obtain pushouts $\fc(X_0,X_{\ell})$ for $\ell = 1, \ldots, n$.
Then use Definition~\ref{def:P+} for each $\ell=1,\ldots,n$, to obtain 
\begin{equation*}
  P: \fc(X_0) \xto{P^+_1} \fc(X_0,X_1) \xto{P^-_2} \fc(X_0,X_2) \xto{P^+_3} \ldots \to \fc(X_0,X_n).
\end{equation*}
Apply Theorem~\ref{thm:vanKampenABX} inductively to obtain the following.

\begin{theorem} \label{thm:vanKampenChain}
  The pushout of compatible chains of future retracts and past retracts is a chain of future retracts and past retracts.
\end{theorem}

It remains to apply this result to extremal models.

\begin{theorem} \label{thm:vanKampenExtremalModels}
  The pushout of compatible compact extremal models in an extremal model.
\end{theorem}

\begin{proof}
  It remains to show that if $P_1: \fc(X_1) \to \fc(X_1,A_1)$ and
  $P_2:\fc(X_2) \to \fc(X_2,A_2)$ are compatible
  extremal models with $\Ext(X_1) \subseteq A_1$ and
  $\Ext(X_2) \subseteq A_2$ then the pushout $P: \fc(X) \to \fc(X,A)$
  (Theorem~\ref{thm:vanKampenChain}) satisfies $\Ext(X) \subseteq A$.

  Let $x \in \Ext(X)$, where $X = \Int(X_2) \cup \Int(X_2)$.  Without
  loss of generality, assume that $x \in X_1$. Then $x \in \Ext(X_1)$
  -- otherwise this would contradict $x \in \Ext(X)$.
  Therefore $P_1(x) \in A_1 \subseteq A$.
\end{proof}

\begin{example} \label{example:vanKampen}
  Let $X_1$ be the subspace of $\dI \times \dI$ obtained by removing the two squares $(0.1,0.3) \times (0.4,0.6)$ and $(0.7,0.9) \times (0.4,0.6)$. Let $X_2$ be the \mbox{d--space} obtained by removing $(0.4,0.6) \times (0.4,0.6)$ from $\dI \times \dI$ and identifying $(0.2,y)$ and $(0.8,y)$ for $y \in [0,1]$. Let $X$ be obtained by identifying $[0.3,0.5] \times \dI$ in $X_1$ with $[0,0.2] \times \dI$ in $X_2$ and $[0.5,0.7] \times \dI$ in $X_1$ with $[0.8,1] \times \dI$ in $X_2$.

Let $P_{1,1}^+$, $P_{1,2}^-$, $P_{2,1}^+$, and $P_{2,2}^-$ be the future retracts and past retracts of $X_1$ and $X_2$ indicated below.
\begin{center}
\psset{unit=2.5cm}
\begin{pspicture}(1,1)
\psset{fillstyle=solid,fillcolor=gray1,linecolor=black}
\psframe(0,0)(1,1)
\psset{fillcolor=white,linecolor=black}
\psframe(0.1,0.4)(0.3,0.6)
\psframe(0.7,0.4)(0.9,0.6)
\psset{linecolor=black,fillstyle=none}
\psline{->}(0.2,0)(0.2001,0)
\psline{->}(0,0.2)(0,0.2001)
\psline[linecolor=green,linestyle=dashed](0.5,0)(0.5,1)
\psdots(0,0)(0.3,0)(0.5,0)(0.5,1)(0.7,1)(1,1)
\psset{linecolor=yellow,linestyle=dotted}
\psline(0,0.4)(0.1,0.4)
\psline(0.3,0)(0.3,0.4)(0.7,0.4)
\psline(0.9,0.4)(1,0.4)
\psline(0.7,0.6)(0.7,1)
\psset{linecolor=blue,linestyle=solid,doubleline=true,doublecolor=gray1,arrowscale =1.5}
\psline{->}(0.35,0.7)(0.45,0.8)
\psline{->}(0.55,0.7)(0.65,0.8)
\psline{->}(0.85,0.7)(0.95,0.8)
\psline{->}(0.15,0.3)(0.05,0.2)
\psline{->}(0.45,0.3)(0.35,0.2)
\psline{->}(0.65,0.3)(0.55,0.2)
\end{pspicture}%
\quad \quad \quad \quad
\begin{pspicture}(1,1)
\psset{fillstyle=solid,fillcolor=gray1,linecolor=black}
\psframe(0,0)(1,1)
\psset{fillcolor=white,linecolor=black}
\psframe(0.4,0.4)(0.6,0.6)
\psset{linecolor=black,fillstyle=none}
\psline{->}(0.4,0)(0.4001,0)
\psline{->}(0,0.2)(0,0.2001)
\psdots(0,0)(0.2,0)(0.8,0)(0.2,1)(0.8,1)(1,1)
\psset{linecolor=green,linestyle=dashed}
\psline(0.2,0)(0.2,1)
\psline(0.8,0)(0.8,1)
\psset{linecolor=yellow,linestyle=dotted}
\psline(0,0.4)(0.4,0.4)
\psline(0.6,0.4)(1,0.4)
\psset{linecolor=blue,linestyle=solid,doubleline=true,doublecolor=gray1,arrowscale =1.5}
\psline{->}(0.05,0.7)(0.15,0.8)
\psline{->}(0.65,0.7)(0.75,0.8)
\psline{->}(0.85,0.7)(0.95,0.8)
\psline{->}(0.15,0.3)(0.05,0.2)
\psline{->}(0.35,0.3)(0.25,0.2)
\psline{->}(0.95,0.3)(0.85,0.2)
\end{pspicture}%
\quad\quad\quad\quad
\begin{pspicture}(0,-0.2)(1,0.8)
\psset{fillstyle=solid,fillcolor=gray1,linecolor=black}
\pspolygon(0,0)(0.8,0)(1,0.4)(0.2,0.4)
\pspolygon[linecolor=gray1](0.31055,0.14472)(0.51055,0.54472)(0.68944,0.45528)(0.48944,0.05528)
\psarc(0.6,0.5){0.1}{-26.565}{153.435}
\pscircle[fillcolor=gray2](0.4,0.1){0.1}
\psline(0.31055,0.14472)(0.51055,0.54472)
\psline(0.68944,0.45528)(0.48944,0.05528)
\psset{fillcolor=white,linecolor=black}
\pspolygon(0.15,0.15)(0.25,0.15)(0.3,0.25)(0.2,0.25)
\pspolygon(0.65,0.15)(0.75,0.15)(0.8,0.25)(0.7,0.25)
\pspolygon[linecolor=white](0.53236,0.33192)(0.58236,0.43192)(0.49080,0.44397)(0.44080,0.34397)
\psline[linecolor=gray1,linewidth=1.2pt](0.44080,0.34397)(0.53236,0.33192)
\psdots[linecolor=gray1,dotsize=1.2pt](0.44080,0.34397)(0.53236,0.33192)(0.59,0.43192)
\psarc[fillcolor=gray1](0.475,0.25){0.1}{53}{112}
\psarc(0.525,0.35){0.1}{53}{112}
\psline(0.53236,0.33192)(0.58236,0.43192)
\psline(0.44080,0.34397)(0.49080,0.44397)
\psset{linecolor=black,fillstyle=none}
\psline{->}(0.2,0)(0.2001,0)
\psline{->}(0.1,0.2)(0.1001,0.2002)
\psset{linecolor=green,linestyle=dashed}
\psline(0.4,0)(0.6,0.4)
\end{pspicture}%
\end{center}
Then $(P_{1,1}^+, P_{1,2}^-)$ and $(P_{2,1}^+, P_{2,2}^-)$ are compatible extremal models. 
Combining $\fc(X_1,A_1)$ and $\fc(X_2,A_2)$, we obtain $\fc(X,A)$. These fundamental categories are generated by the graphs below.
\begin{center}
\psset{unit=2.5cm}
\begin{pspicture}(0,-0.2)(1,1.2)
\psdots(0,0)(0.3,0)(0.5,0)(0.5,1)(0.7,1)(1,1)
\psset{linestyle=dashed,linecolor=red}
\psline(0,0)(0.3,0)(0.5,0)(0.5,1)(0.7,1)(1,1)
\psline(0,0)(0.5,1)
\psline(0.5,0)(1,1)
\psline{->}(0.2,0)(0.201,0)
\psline{->}(0.45,0)(0.451,0)
\psline{->}(0.5,0.6)(0.5,0.602)
\psline{->}(0.65,1)(0.651,1)
\psline{->}(0.9,1)(0.901,1)
\psline{->}(0.3,0.6)(0.301,0.602)
\psline{->}(0.8,0.6)(0.801,0.602)
\end{pspicture}%
\quad \quad \quad \quad
\begin{pspicture}(0,-0.2)(1,1.2)
\psdots(0.3,0)(0.5,0)(0.5,1)(0.7,1)
\psset{linestyle=dashed,linecolor=red}
\psline(0.3,0)(0.5,0)(0.5,1)(0.7,1)
\psline{->}(0.45,0)(0.451,0)
\psline{->}(0.5,0.6)(0.5,0.602)
\psline{->}(0.65,1)(0.651,1)
\pscurve(0.5,0)(0.7,-0.2)(0.8,-0.1)(0.5,0)
\pscurve(0.5,1)(0.3,1.2)(0.2,1.1)(0.5,1)
\psline{->}(0.71,-0.025)(0.711,-0.02515)
\psline{->}(0.39,1.15)(0.389,1.1511)
\end{pspicture}%
\quad \quad \quad \quad
\begin{pspicture}(0,-0.2)(1,1.2)
\psdots(0,0)(0.3,0)(0.5,0)(0.5,1)(0.7,1)(1,1)
\psset{linestyle=dashed,linecolor=red}
\psline(0,0)(0.3,0)(0.5,0)(0.5,1)(0.7,1)(1,1)
\psline(0,0)(0.5,1)
\psline(0.5,0)(1,1)
\psline{->}(0.2,0)(0.201,0)
\psline{->}(0.45,0)(0.451,0)
\psline{->}(0.5,0.6)(0.5,0.602)
\psline{->}(0.65,1)(0.651,1)
\psline{->}(0.9,1)(0.901,1)
\psline{->}(0.3,0.6)(0.301,0.602)
\psline{->}(0.8,0.6)(0.801,0.602)
\pscurve(0.5,0)(0.7,-0.2)(0.8,-0.1)(0.5,0)
\pscurve(0.5,1)(0.3,1.2)(0.2,1.1)(0.5,1)
\psline{->}(0.71,-0.025)(0.711,-0.02515)
\psline{->}(0.39,1.15)(0.389,1.1511)
\end{pspicture}%
\end{center}
\end{example}


\end{document}